\documentclass[10pt]{amsart}
\usepackage{amsmath, amssymb, amsthm,verbatim}
\usepackage{mathtools}
\usepackage[margin=1in,includefoot]{geometry}
\usepackage{array}
\usepackage{tabularx}
\usepackage{mathrsfs}
\usepackage[T1]{fontenc}
\usepackage{hyperref}
\usepackage{color}

\usepackage{xypic}
\usepackage{multicol}

\title{A refined count of Coxeter element reflection factorizations}
\author{Elise del Mas}
\address{School of Mathematics, University of Minnesota, Minneapolis, Minnesota 55455}
\author{Thomas Hameister}
\address{Dept. of Mathematics, University of Wisconsin, Madison, Wisconsin 53706}
\author{Victor Reiner}
\address{School of Mathematics, University of Minnesota, Minneapolis 55455}
\makeatletter
\newtheorem*{rep@theorem}{\rep@title}
\newcommand{\newreptheorem}[2]{%
\newenvironment{rep#1}[1]{%
 \def\rep@title{#2 \ref{##1}}%
 \begin{rep@theorem}}%
 {\end{rep@theorem}}}
\makeatother

\theoremstyle{plain}
\newtheorem{thm}{Theorem}[section]
\newreptheorem{thm}{Theorem}

\newtheorem{prop}[thm]{Proposition}

\theoremstyle{definition}

\theoremstyle{remark}

\theoremstyle{plain}
\newtheorem{cor}[thm]{Corollary}

\def\R{\mathbb R}
\def\N{\mathbb N}
\def\C{\mathbb C}
\def\Z{\mathbb Z}

\def\RR{{\mathcal R}}
\def\reg{{\operatorname{reg}}}
\def\ab{{\operatorname{ab}}}

\newcommand\chihook[2]{\chi_{{}_{#1}\mathfrak{h}^n_{#2}}}

\begin{document}

\begin{abstract}
For well-generated complex reflection groups,
Chapuy and Stump gave a simple product for  
a generating function counting
reflection factorizations of a Coxeter element by 
their length.  This is refined here to
record the number
of reflections used from each orbit of hyperplanes.
The proof is case-by-case via the classification of
well-generated groups.  It implies a new expression for
the Coxeter number, expressed via data coming from
a hyperplane orbit;  a case-free proof 
of this due to J. Michel is included.
\end{abstract}

\maketitle

\section{Introduction}

A {\it complex reflection group} is a finite subgroup $W$ of $GL(V)$, where $V=\C^n$, 
generated by the set of all {\it reflections} $t$ in $W$, that is, the
elements $t$ whose fixed space $V^t:=\ker(t - 1)$ is a {\it hyperplane} $H$, meaning
$\dim H= n-1$.
Let $\RR$ denote the set of all reflections in $W$, and $\RR^*$ the collection of all reflecting hyperplanes.
An important numerological role is played by the cardinalities of $\RR, \RR^*$, denoted $N, N^*$, respectively. 

This paper focusses on the complex reflection groups $W$ which act irreducibly on $V=\C^n$,
and which are {\it well-generated} in the sense that they can be generated by $n$ reflections.  
For such group $W$, one can define the {\it Coxeter number} as $h:=\frac{N+N^*}{n}$, and
then the {\it Coxeter elements} $c$ in $W$ are the elements $c$ which have at least one eigenvector $v$ in $V^\reg:= V \setminus \cup_{H \in \RR^*} H$ with eigenvalue $\zeta_h:=e^{{2 \pi i}{h}}$.
It is known that there is only one conjugacy class of Coxeter elements $c$;
 see, for example, \cite{bessis2006finite, chapuy2014counting}.  

Having fixed one Coxeter element $c$, one can ask for the number 
$f_\ell$ counting reflection factorizations of $c$ having length $\ell$, that is, sequences
$(t_1,\ldots,t_\ell) \in \RR^\ell$ for which 
$
c=t_1 t_2 \cdots t_\ell.
$
The main result of Chapuy and Stump \cite{chapuy2014counting} is the following amazingly 
simple product formula for its exponential generating function:

\begin{equation}
\label{Chapuy-Stump-theorem}
\sum_{\ell \geq 0} f_\ell \frac{x^\ell}{\ell!} = \left( e^{\frac{N}{n}x} - e^{\frac{-N^*}{n}x} \right)^n.
\end{equation}
In particular, this power series starts at $x^n$, because shortest factorizations of $c$ have length $n$.

Our main result refines \eqref{Chapuy-Stump-theorem}, 
accounting for how many reflections $t_j$ appearing in
$c=t_1 t_2 \cdots t_\ell$
have their reflecting hyperplane $\ker(t_j-1)$ lying in
the various $W$-orbits $\RR^*_1,\ldots,\RR^*_p$ decomposing  $\RR^*=\sqcup_{i=1}^p \RR^*_i$.  
Stating it requires some numerology associated to each orbit $\RR^*_i $ for $i=1,2,\ldots,p$. 
Let $\RR_i$ denote the subset of reflections whose reflecting hyperplane lies in $\RR^*_i$,
so $\RR=\sqcup_{i=1}^p \RR_i$. Define
$$
%\begin{aligned}
N_i:=\#\RR^*_i, \quad
N^*_i:=\#\RR_i, \quad
\text{ and }\quad
n_i:=\# \left\{ \begin{matrix} j \in \{1,2,\ldots,n\}: t_j \in \RR_i,  \text{ for any }\\
                          \text{ {\it shortest} factorization } 
                          c=t_1 t_2 \cdots t_n \end{matrix} \right\}
%\end{aligned}
$$
It is not obvious that these numbers $n_i$ are well-defined, independent of the choice of
a length $n$ factorization for $c$, but this follows from work of Bessis \cite[Prop. 7.6]{bessis2006finite}, who showed that any two such shortest factorizations can be connected by
a sequence of {\it Hurwitz moves}
\begin{equation}
\label{Hurwitz-move}
%\begin{array}{rccll}
(t_1,t_2,\ldots, \,\,\, t_k, t_{k+1}   \,\,\,  ,\ldots,t_{n-1},t_n) \longmapsto 
(t_1,t_2,\ldots,  \,\,\,  t_k t_{k+1} t_k^{-1}  ,t_k  \,\,\,  ,\ldots,t_{n-1},t_n) 
%\end{array}
\end{equation}

Let
$
f_{\ell_1,\ell_2,\ldots,\ell_p}
$
be the number of tuples $(t_1,\ldots,t_\ell)$ factoring $c=t_1 t_2 \cdots t_\ell$
having the first $\ell_1$ reflections $t_1,t_2,\ldots,t_{\ell_1}$  in $\RR_1$,
the next $\ell_2$ reflections in $\RR_2$, etc.  (so $\ell=\sum_{i=1}^p \ell_i$).
One can show using the Hurwitz moves above (or see Proposition~\ref{Frobenius} below), 
that 
$
f_{\ell_1,\ell_2,\ldots,\ell_p}
$
also counts factorizations in which the elements of $\RR_i$ occur in {\it any} prescribed set of
the $\ell_i$ positions, rather than all $t_j$ in $\RR_1$ first, then $\RR_2$ second, etc.

\begin{thm}
\label{main-theorem}
For any irreducible, well-generated complex reflection group, and notation as above, one has
%the multivariate exponential generating function of the $f_{\ell_1,\ell_2,\ldots,\ell_p}$ factors as
$$
\sum_{(\ell_1,\ldots,\ell_p) \in \N^p}
  f_{\ell_1,\ldots,\ell_p} \frac{x_1^{\ell_1} \cdots x_p^{\ell_p}}{\ell_1! \cdots \ell_p!}
 =\frac{1}{\#W} \prod_{i=1}^p \left( e^{\frac{N_i}{n_i}x_i} - e^{-\frac{N^*_i}{n_i}x_i} \right)^{n_i}.
$$
\end{thm}

\noindent
Our proof is the same as Chapuy and Stump's proof of \eqref{Chapuy-Stump-theorem}, via the classification\footnote{It should be noted that, at least for
the case of crystallographic real reflection groups (Weyl groups), J. Michel \cite{jm} has also produced a case-free derivation of \eqref{Chapuy-Stump-theorem}, via properties of Deligne-Lusztig representations.}
of irreducible, well-generated reflection groups, and Frobenius's 
character-theoretic technique for counting factorizations, reviewed in Section~\ref{Frobenius-section}.  
Since there is little novelty in the methods, the proof in Section~\ref{proof-section} is abbreviated as much as possible.

One {\bf caveat}: The phrasing of Theorem~\ref{main-theorem}, while convenient, may seem deceptively general, since 
the classification of irreducible complex reflection groups shows that $p=1$ or $2$ in every case. When $p=1$, 
Theorem~\ref{main-theorem} is the same as \eqref{Chapuy-Stump-theorem}, giving
no further information.  The remaining cases where $p=2$ are listed in the table
below, with the factorization in the theorem shown, using variables $(x,y)$ instead of
$(x_1, x_2)$:

%%%%%%%%
\begin{center}
\begin{tabular}{|c|c|c|}\hline \hline
$W$ 
& $\begin{matrix}\text{Coxeter-Shephard}\\\text{diagram}\end{matrix}$
& $\# W \cdot \displaystyle \sum_{(\ell_1,\ell_2)} f_{\ell_1,\ell_2} \frac{x^{\ell_1} y^{\ell_2}}{\ell_1! \ell_2!}$ \\ \hline
$\begin{matrix}G(r,1,n)\\ r \geq 2\end{matrix}$ & 
\xymatrix @-1pc {
*+[Fo]{r}\ar@{-}[r]^{4} & *+[Fo]{2} \ar@{-}[r] & *+[Fo]{2}\ar@{-}[r] &\cdots \ar@{-}[r]&*+[Fo]{2} }
&
$\Big(e^{(r-1)nx} - e^{-nx}\Big)\Big( e^{\frac{nr}{2}y} - e^{-\frac{nr}{2}y} \Big)^{n-1}$
\\ \hline

$\begin{matrix}G(m,m,2) \\ m \geq 4,\text{ even}\end{matrix}$ & 
\entrymodifiers={+[o][F-]}
\xymatrix @-1pc {
2\ar@{-}[r]^{m} & 2 }
&
$\Big(e^{\frac{m}{2}x} - e^{-\frac{m}{2}x}\Big)\Big( e^{\frac{m}{2}y} - e^{-\frac{m}{2}y} \Big)$
\\ \hline

$G_5$ & 
\entrymodifiers={+[o][F-]}
\xymatrix @-1pc {
3\ar@{-}[r]^{4} & 3 }
&
$\Big( e^{8x} - e^{-4x} \Big)\Big( e^{8y} - e^{-4y} \Big)$
\\ \hline

$G_6$ &
\entrymodifiers={+[o][F-]}
\xymatrix @-1pc {
2\ar@{-}[r]^{6} & 3 }
&
$\Big( e^{6x} - e^{-6x} \Big) \Big( e^{8y} - e^{-4y} \Big)$
\\ \hline

$G_9$ &
\entrymodifiers={+[o][F-]}
\xymatrix @-1pc {
2\ar@{-}[r]^{6} & 4 } &
$\Big( e^{12x} - e^{-12x} \Big) \Big( e^{18y} - e^{-6y} \Big)$
\\ \hline

$G_{10}$ &
\entrymodifiers={+[o][F-]}
\xymatrix @-1pc {
3\ar@{-}[r]^{4} & 4 }
&
$\Big( e^{16x} - e^{-8x} \Big)\Big( e^{18y} - e^{-6y} \Big)$
\\ \hline

$G_{14}$ &
\entrymodifiers={+[o][F-]}
\xymatrix @-1pc {
3\ar@{-}[r]^{8} & 2 }
&
$\Big( e^{16x} - e^{-8x} \Big)\Big( e^{12y} - e^{-12y} \Big)$
\\ \hline

$G_{17}$ &
\entrymodifiers={+[o][F-]}
\xymatrix @-1pc {
2\ar@{-}[r]^{6} & 5 }
&
$\Big( e^{30x} - e^{-30x} \Big)\Big( e^{48y} - e^{-12y} \Big)$
\\ \hline

$G_{18}$ &
\entrymodifiers={+[o][F-]}
\xymatrix @-1pc {
3\ar@{-}[r]^{4} & 5 }
&
$\Big( e^{40x} - e^{-20x} \Big)\Big( e^{48y} - e^{-12y} \Big)$
\\ \hline

$G_{21}$ &
\entrymodifiers={+[o][F-]}
\xymatrix @-1pc {
2\ar@{-}[r]^{10} & 3 }
&
$\Big( e^{30x} - e^{-30x} \Big)\Big( e^{40y} - e^{-20y} \Big)$
\\ \hline

$G_{26}$ &
\entrymodifiers={+[o][F-]}
\xymatrix @-1pc {
3\ar@{-}[r] & 3\ar@{-}[r]^4 & 2 }
&
$\Big( e^{12x} - e^{-6x} \Big)^2\Big( e^{9y} - e^{-9y} \Big)$
\\ \hline

$G_{28}$ &
\entrymodifiers={+[o][F-]}
\xymatrix @-1pc {
2\ar@{-}[r] & 2\ar@{-}[r]^{4} & 2\ar@{-}[r] & 2}
&
$\Big( e^{6x} - e^{-6x} \Big)^2\Big( e^{6y} - e^{-6y} \Big)^2$
\\ \hline

\end{tabular}
\end{center}
%%%%%%%%
\noindent
The second column is the {\it Coxeter-Shephard diagram} for these groups, reflecting the case-by-case
observation that irreducible, well-generated groups $W$ with $p =2$ 
are all {\it Shephard groups}, that is, symmetry groups of {\it regular
complex (or real) polytopes}.  This implies (see \cite{Coxeter}) that they have a {\it Shephard presentation}
$$
W = \left \langle S=\{s_1,\ldots,s_n\} \quad | \quad
 s_i^{p_i}=1, \quad \underbrace{s_is_js_is_j \cdots}_{m_{ij}\text{ factors}} = \underbrace{s_js_is_js_i\cdots}_{m_{ij}\text{ factors}} \right\rangle
$$
where here the integer $p_i \geq 2$ labels the node for $s_i$, and the integer $m_{ij} \geq 2$ labels the
edge from $s_i$ to $s_j$, with $m_{ij}=2$ whenever $|i-j| \geq 2$ (and no edge from $s_i$ to $s_j$ is shown).  It is known for Coxeter groups and Shephard groups, 
one can choose $s_1,\ldots,s_n$ so that their
product factors a Coxeter element $c=s_1 s_2 \cdots s_n$.   The hyperplane orbits
$\RR^*_1, \RR^*_2$ correspond to the connected components obtained when one erases the edges with
even labels $m_{ij}$ in the Coxeter-Shephard diagram, and in this case,  $n_1,n_2$ may be re-interpreted as 
the number of nodes in the corresponding connected component.

We also explain (Proposition~\ref{specialization-prop}) 
how Theorem~\ref{main-theorem} necessarily specializes to recover
\eqref{Chapuy-Stump-theorem}.  Comparing the two results then 
gives our first proof  of the following seemingly new fact about
the Coxeter number $h$.

\begin{cor}
\label{hyperplane-orbits-know-h}
For irreducible, well-generated complex reflection groups, and notation as above, 
each hyperplane orbit $\RR^*_i$ for $i=1,2,\ldots,p$ satisfies 
$$
h = \frac{N_i+N^*_i}{n_i}.
$$
\end{cor}

\noindent
Because it uses Theorem~\ref{main-theorem} and \eqref{Chapuy-Stump-theorem},  
this first proof of Corollary~\ref{hyperplane-orbits-know-h}
relies on case-by-case checks.  We also give
a second proof which is case-free, but applies only to real reflection groups, and 
a third proof for the general case supplied by J. Michel, proving a more general assertion 
about regular elements (Theorem~\ref{Michel-result}), which he has kindly allowed
us to reproduce here.

%%%%%%%%%%
\section{Frobenius's method}
\label{Frobenius-section}

Frobenius gave a method, using character theory, for
counting factorizations of an element in any finite group $W$
as a product of elements from specified conjugacy-closed subsets.
Recall that (finite-dimensional, complex) representations $W \overset{\rho}{\rightarrow} GL(V)$
are determined up to equivalence by their character $\chi_\rho: W \rightarrow \C$ defined
by $\chi_\rho(w):=\mathrm{Trace}(V \overset{\rho}{\rightarrow} V)$. For subsets $A \subseteq W$,
define $\chi(A):=\sum_{w \in A} \chi(w)$.

\begin{prop}(Frobenius; see, e.g., \cite[Thm A.1.9]{LandoZvonkin})
\label{Frobenius}
For $A_1,\ldots,A_\ell$ subsets of a finite group $W$, with each $A_i$ closed under conjugation,
and $c$ in $W$, the number of factorizations $c=t_1 \cdots t_\ell$
with $t_i$ in $A_i$ equals
$$
\frac{1}{\# W} \sum_\chi \frac{ \chi(c^{-1}) \chi(A_1) \ldots \chi(A_\ell) }{ \chi(1)^{\ell-1} }
$$
where the sum is over all the characters $\chi$ of the inequivalent irreducible representations of G.
\end{prop}

To apply this here, recall that for Coxeter elements $c$ in $W$ a well-generated complex reflection group,
we defined $f_{\ell_1,\ldots,\ell_p}$ as the number of sequences $(t_1,\ldots,t_\ell)$
factoring $c=t_1 \cdots t_\ell$ in which exactly $\ell_i$ of the factors $t_j$ lie in $\RR_i$,
with the factors from $\RR_1$ all coming first in the sequence, those from $\RR_2$ coming next, etc.

\begin{cor}
\label{relevant-Frobenius-formula}
With the above notations, 
$$
f_{\ell_1,\ldots,\ell_p}
=\frac{1}{\# W} 
   \sum_\chi \frac{ \chi(c^{-1}) \chi(\RR_1)^{\ell_1} \ldots \chi(\RR_p)^{\ell_p }}{ \chi(1)^{\ell-1} }
$$
\end{cor}

%%%%%%%%%%
\section{Proofs of Corollary~\ref{hyperplane-orbits-know-h}.}

Before proving Theorem~\ref{main-theorem}, we explain how it specializes to
\eqref{Chapuy-Stump-theorem}, and why this implies Corollary~\ref{hyperplane-orbits-know-h}.

Note that in each summand on the right in Corollary~\ref{relevant-Frobenius-formula}, the order of the factors 
$\chi(\RR_1)^{\ell_1} \ldots \chi(\RR_p)^{\ell_p }$does not matter.
This explains an assertion from the Introduction: 
$f_{\ell_1,\ldots,\ell_p}$ also counts sequences $(t_1,\ldots,t_\ell)$
factoring $c=t_1 \cdots t_\ell$ in which exactly $\ell_i$ of the factors $t_j$ lie in $\RR_i$,
but where one {\it fixes} any of the $\binom{\ell}{\ell_1,\ldots,\ell_p}$ choices of the positions
in which the factors from $\RR_1, \ldots, \RR_p$ should occur.  This has the following implication.

\begin{prop}
\label{specialization-prop}
The exponential generating functions in Theorem~\ref{main-theorem} and \eqref{Chapuy-Stump-theorem}
are related by specialization:
$$
\sum_{\ell \geq 0} f_\ell \frac{x^\ell}{\ell!}
=\left[ 
\sum_{(\ell_1,\ldots,\ell_p) \in \N^p}
  f_{\ell_1,\ldots,\ell_p} \frac{x_1^{\ell_1} \cdots x_p^{\ell_p}}{\ell_1! \cdots \ell_p!}
\right]_{x_i = x \text{ for }i=1,2,\ldots,p}.
$$
\end{prop}
\begin{proof}
The discussion of the preceding paragraph shows that
$$
f_\ell = \sum_{\substack{(\ell_1,\ldots,\ell_p) \in \N^p:\\ \sum_i \ell_i=\ell}} 
  \binom{\ell}{\ell_1,\ldots,\ell_p} f_{\ell_1,\ldots,\ell_p}
$$
and the rest is simple manipulation of summations and factorials.
\end{proof}

From this one can now see why Theorem~\ref{main-theorem} 
and \eqref{Chapuy-Stump-theorem}
imply Corollary~\ref{hyperplane-orbits-know-h}.

\begin{proof}[First proof of Corollary~\ref{hyperplane-orbits-know-h}]
Plugging \eqref{Chapuy-Stump-theorem} into the left of
Proposition~\ref{specialization-prop} and plugging Theorem~\ref{main-theorem} into the right, gives this equality:
$$
\left( e^{\frac{N}{n}x} - e^{-\frac{N^*}{n}x} \right)^n
=
\prod_{i=1}^p \left( e^{\frac{N_i}{n_i}x} - e^{-\frac{N^*_i}{n_i}x} \right)^{n_i}.
$$
Factoring
$
e^{\frac{N}{n}x} - e^{-\frac{N^*}{n}x}
=e^{-\frac{N^*}{n}x}(e^{\frac{N+N^*}{n}x}-1)
=e^{-\frac{N^*}{n}x}(e^{hx}-1)
$
on the left, and similarly on the right, gives
$$
e^{-N^*x} (e^{hx}-1)^n
=  e^{-x \sum_{i=1}^p N_i^*} 
     \prod_{i=1}^p \left( e^{\frac{N_i+N_i^*}{n_i}x}-1 \right)
$$
On the other hand, by definition, $\sum_{i=1}^p N_i^* = N^*$, and hence
$$
(e^{hx}-1)^n = \prod_{i=1}^p \left( e^{\frac{N_i+N_i^*}{n_i}x}-1 \right).
$$
Then the desired equality 
$\frac{N_i+N_i^*}{n_i}=h$ for $i=1,2,\ldots,p$ follows from this claim:
\begin{quote}
{\bf Claim:} A series $P(x)=\prod_{i=1}^p \left( e^{a_i x}-1 \right)$ in $\R[[x]]$ 
uniquely determines the multiset $(a_1,\ldots,a_p)$.  
\end{quote}
One way to see this claim
is to first write 
$$
P(x) = a_1 \cdots a_p \cdot x^p \prod_{i=1} \frac{e^{a_i x}-1}{a_i x}
= a_1 \cdots a_p x^p + o(x^{p+1})
$$
where the last equality holds since $\frac{e^z-1}{z}$ in $\R[[z]]$ has constant term $1$.
Thus at least the product $a_1 \cdots a_p$ is determined by $P(x)$.  
Naming the coefficients $c_k$ in the unique expansion 
$\log \left( \frac{e^z-1}{z}\right) = \sum_{k=0}^\infty c_k z^k$ in $\R[[z]]$ lets
one read off from $P(x)$ all of the {\it power sums} $\{ a_1^k+\cdots+a_p^k \}_{k=1,2,\ldots}$,
via this calculation:
$$
\begin{aligned}
\log \frac{ P(x) }{a_1 \cdots a_p x^p} 
&= \sum_{i=1}^p \log \left( \frac{e^{a_i x}-1}{a_ix} \right)
=  \sum_{i=1}^p \sum_{k=0}^\infty c_k (a_ix)^k
= \sum_{k=0}^\infty c_k x^k (a_1^k+\cdots+a_p^k).
\end{aligned}
$$
But then these power sums uniquely determine the  multiset $(a_1,\ldots,a_p)$.
\end{proof}

As mentioned in the Introduction, the above first proof of 
Corollary~\ref{hyperplane-orbits-know-h}
relies on Theorem~\ref{main-theorem} and \eqref{Chapuy-Stump-theorem},
both proven via case-by-case arguments.  We therefore seek case-free proofs.
The second proof will apply only when $W$ is a {\it real} reflection group.

\begin{proof}[Second proof of Corollary~\ref{hyperplane-orbits-know-h}, for real $W$, but case-free]
Let $W$ be an irreducible real reflection group, with simple reflections $S=\{s_1,\ldots,s_n\}$,
root system  $\Phi$, and corresponding simple roots $\{\alpha_1,\ldots,\alpha_n\}$.
Then it is known that
the Coxeter element $c=s_1 s_2 \cdots s_n$ generates a cyclic subgroup $\langle c \rangle$ of order $h$
acting {\it freely} on the {\it root system}, decomposing $\Phi=\sqcup_{i=1}^n \Phi_i$ into $n$ orbits 
$\Phi_i$.  Furthermore, one has  $\langle c \rangle$-orbit representatives 
$\theta_j:=s_n s_{n-1} \cdots s_{j+1}(\alpha_j)$, so  that $\theta_j$ is the $W$-orbit of $\alpha_j$; 
see\footnote{The results quoted here assume a {\it crystallographic} root system, but avoid
the crystallographic hypothesis in their proof.}  Bourbaki \cite[Chap. VI, \S 11, Prop. 33]{bourbaki}.  
The factorization $c=s_1 s_2 \cdots s_n$ then implies the first equality here
$$
n_i
\quad =\quad 
\#\{ \alpha_1^\perp, \ldots,\alpha_n^\perp\}  \cap \RR^*_i
\quad=\quad
\#\{ \theta_1^\perp,\ldots,\theta_n^\perp\} \cap \RR^*_i
\quad=\quad
\frac{\#\Phi_i}{h}
\quad =\quad
\frac{2N_i}{h}
\quad =\quad
\frac{N_i+N_i^*}{h},
$$
while the third equality comes from the fact that the $\theta_j$ represent the
orbits for the free $\langle c \rangle$-action on $\Phi$.
\end{proof}

The promised third proof of Corollary \ref{hyperplane-orbits-know-h}, due to J. Michel, 
is case-free and even proves a more general assertion.
Recall that a positive integer $d$ is called a {\it regular number} for $W$
if there is a regular element $w$ in $W$ (one with an eigenvector $v$ in $V^\reg$)
having order $d$.  Recall also that it is a consequence of a characterization of regular 
numbers (originally proven case-by-case by
Lehrer and Springer \cite{LehrerSpringer}, and later in a case-free fashion by
Lehrer and Michel \cite{LehrerMichel}) that the Coxeter number $h$ is a regular number
for every well-generated group.

\begin{thm}(J. Michel)
\label{Michel-result}
A complex reflection group $W$ has every regular number $d$ dividing $N_i+N_i^*$ for
each $i=1,2,\ldots,p$.  In particular, when $W$ is irreducible, well-generated and $d=h$, 
one has $\frac{N_i + N_i^*}{h}=n_i$.
\end{thm}

\noindent
The proof uses the theory of the {\it braid group} $B:=\pi_1(V^\reg/W)$ 
associated to a complex reflection group $W$; see  Brou\'e, Malle, and 
Rouquier \cite{BroueMalleRouquier}, further developments by Bessis \cite{bessis2006finite}, and 
the exposition in Brou\'e \cite{broue2010}.

This theory emphasizes a certain
generating set $\{ s_H\}_{H \in \RR^*}$ for $W$, where $s_H$ is the {\it distinguished reflection}
fixing $H$, the one having $\det(s_H)=\zeta_{\#W_H}$, where $W_H$ is the cyclic
subgroup pointwise fixing $H$.   

Two surjections out of $B$ play an important role here.   First is the surjection 
$B \twoheadrightarrow W$ sending $b \mapsto w$, which arises because the quotient map
$V^\reg \rightarrow V^\reg/W$ is a Galois covering with Galois group $W$;
say that $b$ {\it lifts} $w$ in this situation.  For each hyperplane $H$, there is an important
family of lifts of $s_H$ to elements $s_{H,\gamma}$ in $B$,  called {\it braid reflections};  all
of these braid reflection lifts  $s_{H,\gamma}$ of $s_H$ lie in the same $B$-conjugacy class.

Second is the {\it abelianization} map 
$$
\begin{array}{rcl}
B &\twoheadrightarrow &B^\ab=B/[B,B] \cong \Z^p\\
\gamma & \mapsto & \gamma^\ab.
\end{array}
$$
The composite map $B \rightarrow \Z^p$ can be defined by the following property  \cite[Thm. 2.17]{BroueMalleRouquier}:   if $H$ lies in the $W$-orbit $\RR^*_i$ inside $\RR^*$, then each
braid reflection $s_{H,\gamma}$ lifting $s_H$ maps to the $i^{th}$ standard basis vector of $\Z^p$.

\begin{proof}[Proof of Theorem~\ref{Michel-result}.]
There is a special central element of $B$, denoted $\pi$ in \cite[Not. 2.3]{BroueMalleRouquier}, and called the {\it full twist} $\tau$ in \cite[Def. 6.12]{bessis2006finite}, with the following abelianized image \cite[Lem. 2.22(2), Cor. 2.26]{BroueMalleRouquier}: 
$$
\pi^\ab=( N_1+N_1^*,\ldots,N_p+N_p^*).
$$
%A result of Bessis \cite[Thm. 12.4]{bessis2006finite}, \cite[Thm. 5.33]{broue2010} asserts\footnote{\color{red}Jean, can you check that there is no {\bf hidden} dependence on facts that Bessis proved case-by-case here? One concern I had is that \S 12 of \cite{bessis2006finite} quotes work from his {\tt arXiv} preprint GR/0610778 "Garside categories, periodic loops and cyclic sets", and therefore might rely on the lattice property for $NCP(W)$, which he checked case-by-case, to deduce the Garside property.}  that the regular numbers $d$ are exactly the
%positive integers for which there exists $\rho$ in $B$ with the property that $\rho^d=\pi$;
%in fact, there is at most one $B$-conjugacy class of such $\rho$ for each regular number $d$. 
When $d$ is a regular number, there exists $\rho$ in $B$ with the property
 $\rho^d=\pi$, see \cite[Prop. 5.24]{broue2010}.
Consequently, any such $\rho$ has $\rho^\ab$ satisfying
\begin{equation}
\label{Michel's-observation}
d \cdot \rho^\ab=(\rho^d)^\ab=\pi^\ab=( N_1+N_1^*,\ldots,N_p+N_p^*) \quad
 \text{ in } \Z^p
\end{equation}
proving the first assertion of the theorem, that $d$ divides $N_i+N_i^*$ for $i=1,2,\ldots,p$.

In the special case where $W$ is irreducible, well-generated, and $d=h$, Bessis defined \cite[Def. 6.11]{bessis2006finite}
a certain element  $\delta$ in $B$, which lifts a Coxeter element $c$ in $W$
 (see \cite[Lemmas 6.13, 7.3]{bessis2006finite}),
and which has a factorization $\delta= s_{H_1,\gamma_1} \cdots  s_{H_n,\gamma_n}$ into $n$
braid reflections; see \cite[Rmk. 6.10, Lem. 7.4]{bessis2006finite}.  Thus, by 
our earlier definition of $n_i$, and the aforementioned characterization of the abelianization map, one has 
$$
\delta^\ab=\left( s_{H_1,\gamma_1} \cdots  s_{H_n,\gamma_n} \right)^\ab=(n_1,\ldots,n_p).
$$ 
Consequently, in this case 
\eqref{Michel's-observation} tells us that
$$
 h\cdot (n_1,\ldots,n_p) = h \cdot \delta^{\ab} =( N_1+N_1^*,\ldots,N_p+N_p^*) \quad \text{ in } \Z^d,
$$
showing the desired equality $n_i=\frac{N_i + N_i^*}{h}$ for $i=1,2,\ldots,p$.
\end{proof}

%%%%%%%%%
\section{Proof of Theorem \ref{main-theorem}}
\label{proof-section}

As explained in the caveat following Theorem~\ref{main-theorem}, 
the number $p$ of $W$-orbits of hyperplanes is either $1$ or $2$.  When $p=1$, the theorem is equivalent to \eqref{Chapuy-Stump-theorem}, and so there is nothing further to prove.  
The irreducible, well-generated groups $W$ having $p=2$ appear 
in the table following the caveat, with only two infinite families $ G(m,m,2), G(r,1,n)$, and several exceptional groups.  Just as in \cite{chapuy2014counting}, one can use Frobenius's Proposition~\ref{Frobenius} to verify the
table entries-- we give here the general calculations for the two infinite families in the next two subsections.  
The exceptional cases were handled via computer, accessing in {\tt SAGE} (via
the {\tt Gap3} package {\tt Chevie}, see \cite{chevie}) the irreducible complex reflection groups and their character tables; we discuss the exceptional cases no further here.

%%%%%
\subsection{The dihedral group $G(m,m,2)$ for even $m$.}
The group $G(m,m,2)$ turns out to be the complexification of
a real reflection group, the {\it dihedral group} of type $I_2(m)$ with
Coxeter presentation 
$$
W=\langle s_1,s_2: s_1^2=s_2^2=e=(st)^m \rangle.
$$
Here the sets $\RR^*, \RR$ of reflecting hyperplanes (lines) and reflections
both have size $m$.  Both $\RR^*, \RR$ have a single $W$-orbit
when $m$ is odd, but when $m$ is even,
they decompose two orbits of size $N_1^*=N_1=\frac{m}{2}=N_2=N_2^*$,
indexed here so that $s_i$ lies in $\RR_i$ for $i=1,2$.  Furthermore, $c=s_1 s_2$, and
$n_1=n_2=1$.

Irreducible $W$-representations have dimension one or two,
and for every two-dimensional irreducible character $\chi$,
one has vanishing character values $\chi(s_1)=\chi(s_2)=0$.
Hence only one-dimensional characters $\chi$ contribute in the formula 
Corollary~\ref{relevant-Frobenius-formula} for $f_{\ell_1,\ell_2}$.
For $m$ even, there are four such characters, namely
$\{\mathbf{1}, \,\, \chi_1, \,\, \chi_2, \,\, \chi_1 \chi_2\}$,
with values determined by $\chi_i(s_j)=-1$ if $i=j$ and $\chi_i(s_j)=+1$ if $i \neq j$, for $i,j \in \{1,2\}$.

Using the $p=2$ case of Corollary~\ref{relevant-Frobenius-formula} then gives the following:

$$
\begin{aligned}
\#W f_{\ell_1,\ell_2} 
&= \sum_{\chi} \frac{ \chi(c^{-1}) 
                      \left( \frac{m}{2} \chi(s) \right)^{\ell_1}
                      \left( \frac{m}{2} \chi(t) \right)^{\ell_2} }
                       {\chi(1) ^{\ell_1+\ell_2-1}} 
=  \left( \frac{m}{2} \right)^{\ell_1+\ell_2} 
     \sum_{\chi} \frac{ \chi(c^{-1}) 
                      \chi(s)^{\ell_1}
                       \chi(t)^{\ell_2} }
                       {\chi(1) ^{\ell_1+\ell_2-1}} \\
&= \left( \frac{m}{2} \right)^{\ell_1+\ell_2}
    \left( 1+ (-1)^{\ell_1-1} + (-1)^{\ell_2-1} 
     + (-1)^{\ell_1+\ell_2-2} \right)
 = \left( \frac{m}{2} \right)^{\ell_1+\ell_2}
    \left( 1- (-1)^{\ell_1}\right) \left(1- (-1)^{\ell_2}\right)
\end{aligned}
$$
and hence, in agreement with Theorem~\ref{main-theorem}, one calculates
$$
\begin{aligned}
\#W \sum_{(\ell_1, \ell_2) \in \N^2} 
  f_{\ell_1,\ell_2} \frac{x^{\ell_1} y^{\ell_2}}{\ell_1! \ell_2!}
 &=\sum_{(\ell_1, \ell_2) \in \N^2}
      \left( \frac{m}{2} \right)^{\ell_1+\ell_2}
      \left( 1-(-1)^{\ell_1}\right) \left(1- (-1)^{\ell_2}\right)
      \frac{x^{\ell_1} y^{\ell_2}}{\ell_1! \ell_2!} \\
  &=\left( 
    \sum_{\ell_1 \in \N}
      \left( 1- (-1)^{\ell_1}\right) 
      \left( \frac{m}{2} \right)^{\ell_1} \frac{x^{\ell_1}}{\ell_1!} 
  \right)
  \left( 
    \sum_{\ell_1 \in \N}
      \left( 1- (-1)^{\ell_2}\right) 
      \left( \frac{m}{2} \right)^{\ell_2} \frac{y^{\ell_2}}{\ell_2!} 
  \right)\\
&=  \Big(e^{\frac{m}{2}x} - e^{-\frac{m}{2}x}\Big)\Big( e^{\frac{m}{2}y} - e^{-\frac{m}{2}y} \Big).
\end{aligned}
$$

%%%%%
\subsection{The monomial groups $G(r,1,n)$ for $r \geq 2$}

The group $W = G(r,1,n)$ is the set of $n \times n$ matrices with one nonzero entry in each row and column, and that nonzero entry is an $r^{th}$ root-of-unity in $\C$, a power of the primitive root $\zeta_r=e^{\frac{2\pi i}{r}}$.  
The reflecting hyperplane $W$-orbit decomposition  is
$
\RR^*=\RR_1^* \sqcup \RR_2^*
$
where
$$
\begin{aligned}
\RR_1^*&=  \{ x_i=0: 1 \leq i \leq n\}, \quad \text{ so } N_1^*=n,\\
\RR_2^*&=\{ x_i = \zeta_r^k x_j: 1 \leq i < j \leq n, \text{ and } 0 \leq k \leq r-1 \},
 \quad \text{ so } N_2^*=r\binom{n}{2}.
 \end{aligned}
 $$
The accompanying decomposition 
of the reflections
$\RR=\RR_1 \sqcup \RR_2$   has $\RR_1$ consisting of the 
$N_1=(r-1)n$ reflections that scale one of the $n$ coordinates by $\zeta_r^\ell$ for some $1 \leq \ell \leq r-1$, and fix all other coordinates, while $\RR_2$ is the collection of $N_2=N_2^*=r\binom{n}{2}$ 
order two reflections in each of the hyperplanes of $\RR_2^*$.

To finish the computation, we use the character-theoretic analysis already detailed in \cite[\S 5.3]{chapuy2014counting}.   There the authors show that  the only $W$-irreducible characters $\chi$ which do not vanish on $c^{-1}$ form a two-parameter family denoted
$
\{ \chihook{q}{k}\}
$
where $0 \leq q \leq r-1$ and $0 \leq k \leq n-1$,
with these values:
$$
 \chihook{q}{k}(1)  = \binom{n-1}{k},
 \quad
\chihook{q}{k}(c^{-1}) = (-1)^k \zeta_r^{-q},
\quad
\displaystyle \frac{ \chihook{q}{k}(t) }{ \chihook{q}{k}(1) } = 
\begin{cases} 
\zeta_r^{q\ell} & \text{ if }t \in \RR_1 \text{ and } \det(t)=\zeta_r^\ell,\\
\frac{n-1-2k}{n-1}& \text{ if }t \in \RR_2.\\
\end{cases}
$$
Using the $p=2$ case of Corollary~\ref{relevant-Frobenius-formula}, one has
$$
\begin{aligned}
\#W \cdot f_{\ell_1,\ell_2} 
&=
 \sum_\chi \frac{ \chi(c^{-1}) \chi(\RR_1)^{\ell_1}  \chi(\RR_2)^{\ell_2 }}
               { \chi(1)^{\ell_1+\ell_2-1} }
= \sum_\chi \chi(1) \cdot \chi(c^{-1})\cdot  \left(\frac{\chi(\RR_1)}{\chi(1)}\right)^{\ell_1}
                            \left(\frac{\chi(\RR_1)}{\chi(1)}\right)^{\ell_1}\\
& =\sum_{k=0}^{n-1} \sum_{q=0}^{r-1} \quad \binom{n-1}{k}  \cdot (-1)^k\zeta_r^{-q} \cdot 
            \left( n(\zeta_r^q+\zeta_r^{2q}+\cdots+ \zeta_r^{(r-1)q}) \right)^{\ell_1} 
             \left(\frac{nr(n-1-2k)}{2}\right)^{\ell_2} \\
&=    \left(  n^{\ell_1} \sum_{q=0}^{r-1} \zeta_r^{-q}  
                       \left(  \sum_{\ell=1}^{r-1} \zeta_r^{q\ell} \right)^{\ell_1} 
          \right)
          \left(
              \sum_{k=0}^{n-1}\binom{n-1}{k}(-1)^k\left(\frac{nr(n-1-2k)}{2}\right)^{\ell_2}
          \right)     
\end{aligned}
$$
Note that 
$$
\sum_{\ell=1}^{r-1} \zeta_r^{q\ell} 
=-1+ \sum_{\ell=0}^{r-1} \zeta_r^{q\ell}
=
\begin{cases}
r-1 &\text{ if }q=0,\\
-1& \text{ if  }q=1,2,\ldots,r-1,
\end{cases}
$$
and hence 
$$
\sum_{q=0}^{r-1} \zeta_r^{-q}  
                       \left(     \sum_{\ell=1}^{r-1} \zeta_r^{q\ell}       \right)^{\ell_1}
= (r-1)^{\ell_1} - (-1)^{\ell_1}.
$$
Therefore one can check agreement with Theorem~\ref{main-theorem} as follows:
$$
\#W \sum_{(\ell_1, \ell_2) \in \N^2} 
     f_{\ell_1,\ell_2} \frac{x^{\ell_1} y^{\ell_2}}{\ell_1! \ell_2!}  
 =   \left( \sum_{\ell_1  \in \N} 
    \frac{x^{\ell_1}}{\ell_1!} 
       n^{\ell_1}\left( (r-1)^{\ell_1} - (-1)^{\ell_1}  \right) 
       \right)
     \left( \sum_{\ell_2  \in \N} 
            \frac{y^{\ell_2}}{\ell_2!}   \sum_{k=0}^{n-1}\binom{n-1}{k}(-1)^k\left(\frac{nr(n-1-2k)}{2}\right)^{\ell_2}
     \right)
     $$
in which the first factor on the right is
$$
\sum_{\ell_1 \in \N} 
   \frac{x^{\ell_1}}{\ell_1!} n^{\ell_1} \left( (r-1)^{\ell_1} - (-1)^{\ell_1} \right)
=e^{(r-1)nx}-e^{-nx},
$$
consistent with $n_1=1,N_1=(r-1)n$ and  $N_1^*=n$, while the second factor is
$$
\begin{aligned}
&\sum_{\ell_2 \in \N}
\frac{y^{\ell_2}}{\ell_2!} 
              \sum_{k=0}^{n-1}\binom{n-1}{k}(-1)^k 
                \left(\frac{nr(n-1-2k)}{2}\right)^{\ell_2} \\
&=\sum_{k=0}^{n-1}\binom{n-1}{k}(-1)^k
    \sum_{\ell_2 \in \N} \frac{y^{\ell_2}}{\ell_2!} 
      \left(\frac{nr(n-1-2k)}{2}\right)^{\ell_2} \\
&=\sum_{k=0}^{n-1}\binom{n-1}{k}(-1)^k
    e^{ \frac{nr(n-1-2k)}{2} y} \\
& =\left( e^{\frac{nr}{2}y} \right)^{n-1}  \sum_{k=0}^{n-1}\binom{n-1}{k} ( -e^{-nry} )^k  
=\left( e^{\frac{nr}{2}y} \right)^{n-1} 
      (1- e^{-nry} )^{n-1} 
 = \left( e^{\frac{nr}{2}y} - e^{\frac{-nr}{2}y}\right)^{n-1},
\end{aligned}
$$
consistent with $n_2=n-1$ and $N_2=N_2^*=r\binom{n}{2}=(n-1)\frac{nr}{2}$.

This complete the proof for $W=G(r,1,n)$, and the proof of Theorem~\ref{main-theorem}.

\section*{Acknowledgments} 
Research supported by NSF grants DMS-1148634 and DMS-1601961.
Work of the second author was carried out under the auspices of the 
2017 summer REU program at the School of Mathematics, University of Minnesota, 
Twin Cities.  The authors thank Craig Corsi, Theo Douvropoulos, and Joel Lewis 
for helpful comments, and they thank Jean Michel for allowing them to include his 
proof of Corollary~\ref{hyperplane-orbits-know-h}.

\end{document}